\documentclass[onefignum,onetabnum]{siamart171218}

\usepackage{lipsum}
\usepackage{amsfonts}
\usepackage{graphicx}
\usepackage{epstopdf}
\usepackage{algorithmic}
\ifpdf
  \DeclareGraphicsExtensions{.eps,.pdf,.png,.jpg}
\else
  \DeclareGraphicsExtensions{.eps}
\fi

\newsiamremark{remark}{Remark}
\newsiamremark{hypothesis}{Hypothesis}
\crefname{hypothesis}{Hypothesis}{Hypotheses}
\newsiamthm{claim}{Claim}

\headers{fluorescence Ultrasound Modulated Optical Tomography}{W. Li, Y. Yang and Y. Zhong}

\title{A hybrid inverse problem in the fluorescence ultrasound modulated optical tomography in the diffusive regime\thanks{Submitted to the editors \today.
\funding{The research of YY was partly
supported by NSF Grant DMS-1715178, an AMS Simons travel grant, and a start-up fund from Michigan State University.}}}

\author{Wei Li\thanks{Department of Mathematics, Louisiana State University, Baton Rouge, LA 70803-4918, USA. 
  (\email{liwei@lsu.edu}). \url{} }
\and Yang Yang\thanks{Department of Computational Mathematics, Science and Engineering, Michigan State University, East Lansing, MI 48824, USA. 
  (\email{yangy5@msu.edu}). \url{https://cmse.msu.edu/directory/faculty/yang-yang/} }
\and Yimin Zhong\thanks{Department of Mathematics, University of California, Irvine, Irvine, CA 92697, USA. 
  (\email{yiminz@uci.edu}). \url{} } }

\usepackage{amsopn}

\usepackage{hyperref}
\usepackage{amsmath,amssymb,latexsym,mathrsfs}
\usepackage{enumerate}
\usepackage{bm,cite}
\usepackage{graphicx,color}

\newcommand{\eps}{\varepsilon}

\newcommand{\pdr}[2]
{\dfrac{\partial{#1}}{\partial{#2}}}

 \newcommand{\bbN}{\mathbb N}
 
\newcommand{\bbR}{\mathbb R}

\newcommand{\bn}{\mathbf n}
 \newcommand{\bp}{\mathbf p}
\newcommand{\bq}{\mathbf q} 
\newcommand{\bs}{\mathbf s} 
  
 \newcommand{\bx}{\mathbf x} 
 \newcommand{\bz}{\mathbf z}

\ifpdf
\hypersetup{
  pdftitle={fluorescence Ultrasound Modulated Optical Tomography},
  pdfauthor={W. Li, Y. Yang and Y. Zhong}
}
\fi

\begin{document}

\maketitle

\begin{abstract}
	We investigate a hybrid inverse problem in fluorescence ultrasound modulated optical tomography (fUMOT) in the diffusive regime. We prove that the absorption coefficient of the fluorophores at the excitation frequency and the quantum efficiency coefficient can be uniquely and stably reconstructed from boundary measurement of the photon currents, provided that some background medium parameters are known.
Reconstruction algorithms are proposed and numerically implemented as well.	
\end{abstract}



\section{Introduction}

Fluorescence optical tomography (FOT) is an emerging high-contrast biomedical imaging modality that localizes fluorescent targets within tissues~\cite{arridge1999optical,chang1995fluorescence,chaudhari2005hyperspectral, corlu2007three,milstein2003fluorescence}. In a FOT experiment, a short near-infrared laser pulse of certain excitation frequency is sent to biological samples which are labeled with fluorophores either by injection of dye solutions or by intrinsic gene expression of fluorescent proteins. The fluorophores get excited after absorbing the illumination, and later decay back to the ground state by emitting light at a lower emission frequency.
The emitted light and the residual exicitation light are detectable at the boundary of the tissue, and can be used to reconstruct the spatial concentration and lifetimes of the fluorophores. These reconstructed quantities prove to be useful in many areas including disease diagnosis, gene information decoding and biological processes tracking~\cite{ntziachristos2006fluorescence,o1996fluorescence,stephens2003light}.
%

Despite high optical contrast, the reconstruction in FOT is known to be unstable and suffer from limited resolution at large spacial scales in highly scattering media~\cite{arridge2009optical,bal2009inverse,uhlmann2009electrical,chaudhari2005hyperspectral,herve2010non,leblond2009early,lin2012temperature}.
A strategy to overcome such limitations is to incorporate acoustic modality of high resolution. Two types of incorporation are prevalent. One is to take advantage of the \textit{photo-acoustic effect}--the phenomenon that absorption of near-infrared light causes thermoelastic expansion of the medium which in turn induces acoustic pulses. The resulting imaging modality is known as fluorescence photo-acoustic tomography (fPAT)~\cite{ren2013quantitative,burgholzer2009photoacoustic,razansky2009multispectral,razansky2007hybrid,wang2012photoacoustic,wang2010integrated} . The other is to perturb the medium by acoustic radiation while making optical measurements at the boundary. The resulting imaging modality is called \textit{fluorescence ultrasound modulated optical tomography}, or \textit{fUMOT} in short. This article is devoted to the study of an inverse problem in fUMOT. 

It is worth pointing out that, in the absence of fluorescence, the above two ways of incorporating acoustic information into optical measurement give rise to imaging modalities known as photoacoustic tomography (PAT)~\cite{wang2004ultrasound,bal2011multi,bal2011quantitative,fisher2007photoacoustic,xia2014photoacoustic,wang2016practical,li2017single}.
 and ultrasound-modulated optical tomography (UMOT)~\cite{bal2010inverse,marks1993comprehensive,kempe1997acousto,granot2001detection,wang1995continuous,ku2005deeply,jiao2002two}. , respectively. Other frequent aliases are thermoacoustic tomography (TAT) for the former and acousto-optic imaging (AOI) for the latter. More broadly, the idea of combining high-contrast modalities with high-resolution ones leads to a large variety of novel imaging modalities known as hybrid modalities. Inverse problems in hybrid modalities are often referred to as \textit{hybrid inverse problems} or \textit{coupled physics inverse problems}.

fUMOT is a hybrid imaging modality where FOT is combined with acoustic modulation to produce high-contrast high-resolution images. The principle is to perform multiple measurements of the excitation and emission photon currents at the boundary as the optical properties undergo a series of perturbations by acoustic radiation. These measurements turn out to provide internal information of the optical field, leading to a well-posed inverse problem with resolution at the order of the acoustic wavelength~\cite{bal2009inverse}. The feasiblity of fUMOT has been experimentally verified in ~\cite{yuan2009ultrasound,yuan2008mechanisms,yuan2009microbubble,liu2014ultrasound}.

We now derive the mathematical model of fUMOT in a highly-scattering medium following the derivation of fPAT in~\cite{ren2013quantitative} and that of UMOT in \cite{bal2010inverse}. Let $u(\bx,t)$ be the photon density at the excitation frequency, and $w(\bx,t)$ be the photon density at the emission frequency. Propagation of photons in a highly-scattering medium is typically described by the diffusion equation. As photons in the excitation process and those in the emission process exist simultaneously, their density functions obey the coupled time-dependent diffusion equations:
  \begin{equation}\nonumber\label{eq:fltime}
    \begin{aligned}
      \frac{1}{c} \pdr{u(\bx,t)}{t} - \nabla \cdot D_x(\bx)\nabla u(\bx,t) + (\sigma_{x,a}(\bx) + \sigma_{x,f}(\bx))u(\bx,t) &= 0 \quad\quad \text{ in } \Omega \times \bbR^+,\\ 
     \frac{1}{c} \pdr{w(\bx,t)}{t}  - \nabla \cdot D_m(\bx)\nabla w(\bx,t) + (\sigma_{m,a}(\bx) + \sigma_{m,f}(\bx))w(\bx,t) &= S(\bx,t) \text{ in } \Omega \times \bbR^+,\\
      u(\bx,t) = g(\bx,t), \quad w(\bx,t) = 0 &\;\;\quad\quad\quad \text{ on } \partial\Omega \times \bbR^+,\\
            u(\bx,t) = 0, \quad w(\bx,t) = 0 &\;\;\quad\quad\quad \text{ on } \Omega\times \left\{0\right\}.
    \end{aligned}
  \end{equation}
Here $\Omega\subseteq\bbR^d (d=2,3)$ is the domain of interest, $g(\bx,t)$ is the external excitation source, 
$D_x$ (resp. $D_m$) is the diffusion coefficient, $\sigma_{x,a}$ (resp. $\sigma_{m,a}$) is the absorption coefficient of the medium, and $\sigma_{x,f}$ (resp. $\sigma_{m,f}$) is the \textit{absorption coefficient of the fluorophores at the excitation frequency} (resp. emission frequency). The emission source $S(\bx,t)$ is given by
\begin{equation} \label{eq:intsource}
S(\bx,t) = \eta(\bx) \sigma_{x,f}(\bx) \int_0^t \frac{1}{\tau} e^{-\frac{t-s}{\tau}} u(\bx,s) ds,
\end{equation}
where $\eta(\bx)$ is the \textit{quantum efficiency coefficient} of the fluorophores and $\tau$ is the lifetime of the excited state.
Under the assumptions that the medium is non-trapping and that the lifetime of the fluorophores is greater than the absorption time scale, i.e., $ \sigma_{x,a}(\bx) + \sigma_{x,f}(\bx) >\frac{1}{c\tau}$, the integrals $u(\bx) := \int_{\bbR^+} u(\bx,t)dt$, $w(\bx) := \int_{\bbR^+} w(\bx,t)dt$ and $g(\bx) := \int_{\bbR^+} g(\bx,t)dt$ satisfy the coupled time-independent diffusion equations~\cite{ren2013quantitative}:
  \begin{equation}\label{eq:fl}
    \begin{aligned}
        - \nabla \cdot D_x(\bx)\nabla u(\bx) + (\sigma_{x,a}(\bx) + \sigma_{x,f}(\bx))u(\bx) &= 0, &\quad\text{ in }&\Omega\\ 
      - \nabla \cdot D_m(\bx)\nabla w(\bx) + (\sigma_{m,a}(\bx) + \sigma_{m,f}(\bx))w(\bx) &= \eta(\bx) \sigma_{x,f}(\bx)u(\bx), &\quad\text{ in }&\Omega\\
      u(\bx) = g(\bx), &\quad w(\bx) = 0,\quad&\text{ on }&\partial\Omega
    \end{aligned}
  \end{equation} 
In practice, the coefficient $\sigma_{m,f}$ is extremely small compared to other coefficients hence will be neglected hereafter.

We shall take into account the acoustic modulation for FOT. Consider a plane acoustic wave field of the form $p(\bx, t) = A\cos(\omega t)\cos(\bq \cdot \bx + \phi)$ where $A$ is the amplitude, $\omega$ is the frequency, $\bq$ is the wave vector and $\phi$ is the phase. The acoustic field is supposed to be weak in the sense that $0<A\ll \rho c_s^2$ with $\rho$ the particle number density and $c_s$ the sound speed. The time scale of the acoustic field propagation is generally much greater than that of the optical field and the lifetime of the fluorophores, therefore the acoustic field effectively modulates the time independent diffusion equations~\eqref{eq:fl}. The modulated diffusion and absorption coefficients are shown to take the form~\cite{bal2010inverse}
\begin{equation}\label{eq:mo}
\begin{aligned}
D_x^{\epsilon}(\bx) &= (1 +\epsilon\gamma_x\cos(\bq \cdot \bx + \phi))D_x(\bx), \quad &\gamma_x &= (2 n_x - 1),\\
D_m^{\epsilon}(\bx) &= (1 + \epsilon\gamma_m\cos(\bq \cdot \bx + \phi))D_m(\bx), \quad &\gamma_m &= (2 n_m - 1),\\
\sigma_{x,a}^{\epsilon}(\bx) &= (1 + \epsilon\beta_x\cos(\bq \cdot \bx + \phi) )\sigma_{x,a}(\bx),\quad &\beta_x &= (2 n_x + 1),\\
\sigma_{m,a}^{\epsilon}(\bx)&=(1+\epsilon\beta_m\cos(\bq \cdot \bx + \phi))\sigma_{m,a}(\bx),&\beta_m& = (2 n_m + 1),\\
\sigma_{x,f}^{\epsilon}(\bx) &=(1+\epsilon\beta_f\cos(\bq \cdot \bx + \phi))\sigma_{x,f}(\bx),\quad &\beta_f &= (2 n_f + 1),
\end{aligned}
\end{equation}
where the parameters $\epsilon=\frac{A \cos(\omega t)}{\rho c_s^2}\ll 1$, and $n_x, n_m, n_f$ are the elasto-optical constants of the background medium and fluorophores. Note that the quantum efficiency coefficient $\eta(\bx)$ in~\eqref{eq:intsource} is not altered by the acoustic modulation. 
Combining \eqref{eq:fl} and \eqref{eq:mo}, we obtain the \textit{governing equations for fUMOT}:
\begin{align}
  - \nabla \cdot D^{\epsilon}_{x} \nabla u_{\epsilon} + (\sigma^{\epsilon}_{x,a} + \sigma^{\epsilon}_{x,f})u_{\epsilon} &= 0 & \text{ in } \Omega, &  \quad\quad u_\epsilon = g \quad\quad \text{ on } \partial\Omega; \label{eq:fl2u}\\ 
  - \nabla \cdot D^{\epsilon}_m\nabla w_{\epsilon} + \sigma^{\epsilon}_{m,a}w_{\epsilon} & = \eta \sigma^{\epsilon}_{x,f}u_{\epsilon} & \text{ in } \Omega & \quad\quad w_\epsilon = 0 \quad\quad \text{ on } \partial\Omega. \label{eq:fl2w}
\end{align}
This is the diffusive model that we will be working with in this article.

For a fixed excitation source $g$, the measurement in fUMOT is the boundary photon currents at both the excitation and emission frequencies~\cite{arridge2009optical,ren2010recent}, that is $D_x^{\epsilon} \frac{\partial u_{\epsilon}}{\partial \bn}$ and $D_m^{\epsilon} \frac{\partial w_{\epsilon}}{\partial\bn}$ on $\partial\Omega$ for sufficiently small $\epsilon \geq 0$. Such boundary photon currents can be measured for multiple wave vectors $\bq$ and phases $\phi$. We therefore model the measurement in full generality by the operator
\begin{equation} \label{eq:meas}
\Lambda^{\epsilon}(\bq, \phi) =  \left. (D_x^{\epsilon}\frac{\partial u_{\epsilon}}{\partial \bn} , D_m^{\epsilon} \frac{\partial w_{\epsilon}}{\partial\bn}) \right|_{\partial\Omega}.
\end{equation}
Our goal is to reconstruct the absorption coefficient $\sigma_{x,f}(\bx)$ and the quantum efficiency coefficient $\eta(\bx)$ from $\Lambda^{\epsilon}$, under the assumption that the elasto-optical coefficients and the background parameters $D_x$, $D_m$, $\sigma_{x,a}$ and $\sigma_{m,a}$ have already been reconstructed through other imaging techniques, such as those in~\cite{bal2010inverse,bal2011multi,ren2015inverse,ren2013quantitative}.
%
%

We will tackle this hybrid inverse problem in two steps. Firstly, a non-linear inverse medium problem is solved at the excitation frequency to reconstruct the absorption coefficient $\sigma_{x,f}(\bx)$. Secondly, a linear inverse source problem is studied at the emission frequency to recover the quantum coefficient $\eta(\bx)$. 
Our results for the inverse medium problem include some of the existing results in~\cite{triki2010uniqueness,bal2011non,bal2013cauchy} as sub-cases.

The rest of the paper is organized as follows. In Section~\ref{SEC:ASSUMP}, we state the necessary regularity assumptions to set up the inverse problem, and derive two pieces of internal data from the measurement operator $\Lambda^\epsilon$. Section~\ref{SEC:REC1}, \ref{SEC:REC1_2}, \ref{SEC:NUM} deals with the uniqueness, stability, and reconstruction procedures of the absorption coefficient $\sigma_{x,f}$, respectively. In Section~\ref{SEC:REC1}, we relate the reconstruction of $\sigma_{x,f}$ to the solvability of a semi-linear elliptic boundary value problem and prove the unique identifiability. The stability estimate is derived in Section~\ref{SEC:REC1_2}, which shows the inverse problem of recovering $\sigma_{x,f}$ from the internal data is well posed. Section~\ref{SEC:NUM} consists of three reconstruction algorithms designed for different values of the elasto-optical constants. In Section~\ref{SEC:REC2}, we show that the quantum efficiency coefficient $\eta$ can be uniquely and stably determined by solving an inverse source problem. The proposed reconstruction procedures for both $\sigma_{x,f}$ and $\eta$ are numerically implemented in Section~\ref{SEC:EXP}.

%
%

\section{Internal data}\label{SEC:ASSUMP}
To ensure sufficient regularity of the solutions to the modulated diffusion equations~\eqref{eq:fl2u} \eqref{eq:fl2w}, we make the following assumptions throughout the article:
\begin{enumerate}
	\item [A-1] The domain $\Omega$ is simply connected and $\partial\Omega$ is $C^2$.
	\item [A-2] The optical coefficients satisfy $(D_{x}, \sigma_{x,a})\in C^{0,1}(\overline{\Omega})\times L^{\infty}(\Omega)$, $(D_{m},\sigma_{m,a})\in C^{0,1}(\overline{\Omega})\times L^{\infty}(\Omega)$, $\sigma_{x,f}\in L^{\infty}(\Omega)$, $\eta \in L^{\infty}(\Omega)$, $g(\bx)$ is the restriction of a $C^1(\overline{\Omega})$ function-called $g(\bx)$ again after abusing the notation-on $\partial\Omega$ and 
	\begin{equation}
	K_0 < D_x, D_m, \sigma_{x,a} , \sigma_{m,a}, \sigma_{x,f}, g < K_1
	\end{equation}
	for some positive constants $K_0$ and $K_1$. 
	
\end{enumerate}
Under these assumptions, the elliptic equations~\eqref{eq:fl2u} \eqref{eq:fl2w} admits a unique solution $(u^{\epsilon},w^{\epsilon})\in H^1(\Omega)\times H^1(\Omega)$ by elliptic PDE theory, and the measurement operator 
$$
\Lambda^{\epsilon}:\bbR^d\times \{0,\frac{\pi}{2}\} \to H^{-1/2}(\partial\Omega)\times H^{-1/2}(\partial\Omega), \quad\quad\quad (\bq,\phi)\mapsto \left. (D_x^{\epsilon}\frac{\partial u_{\epsilon}}{\partial \bn} , D_m^{\epsilon} \frac{\partial w_{\epsilon}}{\partial\bn}) \right|_{\partial\Omega}
$$
is continuous.

%
%

For later reference, we derive two types of internal data from the map $\Lambda^\epsilon$. Let $u_{\epsilon}$ and $u_{-\epsilon}$ be solutions to~\eqref{eq:fl2u} with the same boundary condition $g$. 
Following ~\cite{bal2013cauchy}, we integrate by parts to obtain
\small
\begin{equation}\nonumber
\int_{\Omega} (D_x^{\epsilon} - D_x^{-\epsilon})\nabla u_{\epsilon} \cdot \nabla u_{-\epsilon} + (\sigma^{\epsilon}_{x,a} + \sigma^{\epsilon}_{x,f} - \sigma^{-\epsilon}_{x,a} - \sigma^{-\epsilon}_{x,f})u_{\epsilon}u_{-\epsilon} d\bx = \int_{\partial\Omega} D_x^{\epsilon}\frac{\partial u_{\epsilon}}{\partial \bn} g- D_x^{-\epsilon}\frac{\partial u_{-\epsilon}}{\partial\bn} g d\bs
\end{equation}
\normalsize
The right hand side is known since the Dirichlet data $g$ is artificially imposed and the Neumann data $D_x^{\epsilon}\partial_{\bn} u_{\epsilon}$ and $D_x^{-\epsilon}\partial_{\bn} u_{-\epsilon}$ are measured at the boundary $\partial\Omega$. We define
\begin{equation}\nonumber
J_{\epsilon} = \frac{1}{2}\int_{\partial\Omega} D_x^{\epsilon}\frac{\partial u_{\epsilon}}{\partial \bn} g- D_x^{-\epsilon}\frac{\partial u_{-\epsilon}}{\partial\bn} g d\bs
\end{equation}
and perform the asymptotic expansion in $\epsilon$: 
$J_{\epsilon} = \epsilon J_1 + \epsilon^2 J_2 + \dots$. 
Equating the order $\epsilon$ terms and using \eqref{eq:mo}, we find
$
J_1(\bq, \phi) = \int_{\Omega} Q(\bx) \cos(\bq\cdot \bx + \phi) d\bx.
$ 
with
\begin{equation}\label{eq:hamilton1}
Q(\bx):=\gamma_x D_x  |\nabla u_0|^2 + (\beta_x\sigma_{x,a}+\beta_f\sigma_{x,f}) |u_0|^2. 
\end{equation}
The function $Q$ can be recovered from $J_1$ by $Q = \mathcal{F}^{-1}\left(J_1(\bq, 0) - iJ_1(\bq, \frac{\pi}{2})\right)$ where $\mathcal{F}^{-1}$ denotes the inverse Fourier transform. This is our first internal data.

Next, let $v$ and $p$ be auxiliary functions satisfying
\begin{equation} \label{eq:v}
-\nabla \cdot D_m \nabla v + \sigma_{m,a} v = 0\quad\text{ in }\Omega,\quad\quad\quad v= h\quad\text{ on }\partial\Omega 
\end{equation}
\begin{equation} \label{eq:p}
-\nabla \cdot D_x \nabla p + (\sigma_{x,a} + \sigma_{x,f}) p= \eta\sigma_{x,f}v\quad\text{ in }\Omega,\quad\quad\quad p= 0\quad\text{ on }\partial\Omega 
\end{equation}
where $h>0$ is a fixed known function, For simplicity, we assume $h$ is the restriction of a $C^2$ function on $\partial\Omega$ so that $v > 0$ in $\Omega$. Multiply \eqref{eq:fl2u} by $p$, \eqref{eq:p} by $u_\epsilon$, and integrate the difference of the resulting two equations to get
\begin{equation} \label{eq:pu}
\int_\Omega (D_x - D^\epsilon_x) \nabla p \cdot \nabla u_\epsilon + (\sigma_{x,a} + \sigma_{x,f} - \sigma^\epsilon_{x,a} - \sigma^\epsilon_{x,f}) p u_\epsilon \, d\bx = \int_\Omega \eta \sigma_{x,f} v u_\epsilon \, d\bx+ b.t.
\end{equation}
where ``b.t.'' represents various boundary integrals from integration by parts. The integrands in these ``b.t.'' terms here and below involve merely boundary values of $u_\epsilon, w_\epsilon, v, p$ as well as their boundary normal derivatives, hence can be computed from our measurement. Likewise, integrating the difference of \eqref{eq:fl2w} multiplied by $v$ and \eqref{eq:v} multiplied by $w_\epsilon$ gives
\begin{equation} \label{eq:vw}
\int_\Omega (D^\epsilon_m - D_m) \nabla v \cdot \nabla w_\epsilon + (\sigma^\eps_{m,a} - \sigma_{m,a}) v w_\epsilon \, d\bx = \int_\Omega \eta \sigma_{x,f} v u_\epsilon \, d\bx+ b.t..
\end{equation}
Subtract \eqref{eq:vw} from \eqref{eq:pu} to eliminate $\int_\Omega \eta \sigma_{x,f} v u_\epsilon \,d\bx$, then the right hand side involves only boundary integrals, and the order $\epsilon$ term on the left, in view of \eqref{eq:mo}, takes the form $\int_\Omega S(\bx) \cos(\bq \cdot \bx + \phi)$ where
\begin{equation} \label{eq:S}
S := \gamma_m D_m \nabla w_0 \cdot \nabla v + \beta_m \sigma_{m,a} w_0 v + \gamma_x D_x \nabla p \cdot \nabla u_0 + (\beta_x \sigma_{x,a} + \beta_f \sigma_{x,f}) p u_0 - \eta \beta_f \sigma_{x,f} u_0 v.
\end{equation}
Varying $\bq$ and $\phi$ as in \eqref{eq:hamilton1} extracts $S$. This is our second internal data.

%

\section{Reconstruction of \texorpdfstring{$\sigma_{x,f}$}: uniqueness}\label{SEC:REC1}
Our objective is to reconstruct the coefficient $\sigma_{x,f}$ from knowledge of the internal data $Q(\bx)$. This is a nonlinear inverse medium problem.
The concentration of this article will be on the case $\beta_f \neq 0$ (see~\eqref{eq:mo} for the definition of $\beta_f$), but we would like to make a remark here upon the reconstruction if $\beta_f = 0$. Indeed, when $\beta_f = 0$, the internal data $Q$ in~\eqref{eq:hamilton1} does not explicitly contain $\sigma_{x,f}$. However, one can view \eqref{eq:hamilton1} as a Hamilton-Jacobi equation, and prove that it has a unique viscosity solution $u_0$ provided the background coefficients are sufficiently favorable~\cite[Theorem III]{crandall1983viscosity}. One then obtains from~\eqref{eq:fl2u} (with $\epsilon=0$) that
$$
\sigma_{x,f} = \frac{\nabla \cdot D_x\nabla u_0 - \sigma_{x,a}u_0}{u_0}.
$$
This gives an explicit reconstruction in the case $\beta_f = 0$.

We will henceforth study the reconstruction under the assumption $\beta_f\neq 0$. In this situation, determining $\sigma_{x,f}$ from $Q$ reduces to solving a semilinear elliptic equation. 
To see this, multiply the diffusion equation
\begin{equation}\label{eq:u0}
\begin{aligned}
-\nabla\cdot D_x\nabla u_0 + (\sigma_{x,a}+\sigma_{x,f}) u_0 &= 0,\quad\text{ in }\Omega\\
u_0 &= g,\quad \text{ on }\partial\Omega
\end{aligned}
\end{equation}
by $-u_0$ and replace the resulting term $\sigma_{x,f}|u_0|^2$ by the expression in \eqref{eq:hamilton1} to obtain
\begin{equation} \label{eq:u01}
u_0 \nabla \cdot D_x \nabla u_0 + \frac{\gamma_x}{\beta_f} D_x |\nabla u_0|^2 + \sigma_{x,a}(\frac{\beta_x }{\beta_f}-1)|u_0|^2 - \frac{Q}{\beta_f} = 0.
\end{equation}
Set $\tau:=\frac{\gamma_x}{\beta_f}$ and $\mu:=\frac{\beta_x}{\beta_f}-1$. Using the identity $(\tau+1)(u_0 \nabla \cdot D_x \nabla u_0 + \tau D_x |\nabla u_0|^2) = \frac{\nabla\cdot D_x \nabla u^{\tau+1}_0}{u^{\tau-1}_0} $, we combine the first two terms in \eqref{eq:u01} to have
$$
\nabla \cdot D_x \nabla u^{\tau+1}_0 + (\tau+1) \sigma_{x,a} \mu u^{\tau+1}_0 - (\tau+1) \frac{Q}{\beta_f} u_0^{\tau+1} = 0.
$$
Introduce 
\begin{equation} \label{eq:theta}
\theta:=\frac{1-\tau}{1+\tau} = \frac{\beta_f - \gamma_x}{\beta_f + \gamma_x}, \quad\quad\quad \Psi:=u^{\tau+1}_0 =u^{\frac{2}{1+\theta}}_0, 
\end{equation}
then $\Psi$ satisfies the boundary value problem
\begin{equation}\label{eq:key}
\begin{aligned}
\nabla \cdot a \nabla \Psi &= f(\bx, \Psi),\quad&\text{ in }&\Omega\\
\Psi &= g^{\frac{2}{1+\theta}},\quad&\text{ on }&\partial\Omega
\end{aligned}
\end{equation}
where $f(\bx, z) = bz + c|z|^{-(1+\theta)}z$ and the coefficients $a,b,c$ are defined by
\begin{equation}\nonumber
a := D_x , \quad b := -\frac{2}{1+\theta} \sigma_{x,a} \mu, \quad c := \frac{2}{1+\theta}\frac{Q}{\beta_f}.
\end{equation}
Here we have used the absolute value in $f(\bx, z)$ since we only look for positive solutions. It is clear that once $\Psi$ is determined, one can recover $u_0$ and then solve for $\sigma_{x,f}$ from \eqref{eq:hamilton1}. We henceforth focus on the well-posedness of the semilinear elliptic boundary value problem \eqref{eq:key}. We assume throughout this paper that $c \in L^{\infty}(\Omega)$. The two cases $-1\neq \theta<0$ and $\theta \geq 0$ will be treated separately for the uniqueness. Note that the condition $\beta_f \neq 0$ implies $\theta \neq -1$.

\begin{remark}
The case $\theta=\infty$, which corresponds to $\tau=-1$, will not be discussed in this work. However, we would like to point out its connection with the Yamabe's problem in Riemannian geometry. Indeed, dividing \eqref{eq:u01} by $u^2_0$ and introducing $\Phi:=log\, u_0$ yield
$$
\nabla \cdot D_x \nabla \Phi = -\sigma_{x,a} \mu + \frac{Q}{\beta_f} e^{-2\Phi}.
$$ 
This is the Yamabe equation, see~\cite{escobar1992yamabe, lee1987yamabe, taylor1997partial} for further details. 

\end{remark}

\subsection{Case (1): $-1 \neq \theta < 0$}

When $\theta \neq -1$ and $\theta < 0$,  we will prove the boundary value problem ~\eqref{eq:key} has a unique non-negative solution under appropriate assumptions. More precisely, we show

\begin{proposition}\label{pr:1}
Suppose the coefficients $\theta, b, c$ satisfy either one of the following conditions:
	\begin{enumerate}
		\item $-1 <\theta < 0$, and $b \geq 0, c \ge 0$ a.e. in $\Omega$;	
		\item $\theta < -1$, and $b \geq 0, c\ge 0$ a.e. in $\Omega$.
	\end{enumerate}
Then the equation~\eqref{eq:key} admits a unique non-negative weak solution $\Psi\in H^1(\Omega)$ with $\sup_{\Omega}\Psi = \sup_{\partial\Omega}\Psi$. In addition, 
	\begin{enumerate}
		\item when $-1 <\theta < 0$, there exists a constant $g_0 > 0$ such that if the boundary condition in~\eqref{eq:key} satisfies $g > g_0$ on $\partial\Omega$, then $\Psi > 0$ a.e. in $\Omega$.
		\item when $\theta <-1$, then $\Psi > 0$ a.e. in $\Omega$.
	\end{enumerate}
\end{proposition}

The proof will be divided into several lemmas. Firstly, we show a weak solution exists and is unique.
  
\begin{lemma}\label{lm:1}
Under the same hypotheses of Proposition \ref{pr:1}, the equation~\eqref{eq:key} admits a unique weak solution $\Psi\in H^1(\Omega)$.
\end{lemma}
\begin{proof}
We define a variation functional $I[w]$ on the function space \\
$\mathcal{W}=\{w|w\in H^{1}(\Omega) \text{ and }w|_{\partial\Omega}= g^{\frac{2}{1+\theta}}\}$ as follows:
\begin{equation}
I\left[w\right] := \int_{\Omega} L(\bx, w, \nabla w)d\bx = \int_{\Omega} \left(\frac{1}{2}a|\nabla w|^2 + \frac{b}{2} w^2 + 
\frac{c}{1-\theta}|w|^{-(1+\theta)}w^2 \right)d\bx,
\end{equation} 
where $L(\bx, z, \bp)$ denotes the Lagrangian for~\eqref{eq:key}. One can verify that $I\left[w\right]:\mathcal{W}\to\bbR$ is strictly convex and differentiable at $w\in\mathcal{W}$ since $b \geq 0$ and $c \geq 0$, and its derivative is given by
\begin{equation}
I\left[w\right]' v = \int_{\Omega} \left(a\nabla w\cdot \nabla v +bwv +c|w|^{-(1+\theta)}wv\right) d\bx.
\end{equation}
Let $\ell = \max(2, 1-\theta)$, the Lagrangian $L$ satisfies following growth conditions:
\begin{equation}
\begin{aligned}
|L(\bx,\bz,\bp)|&\le K_2 (1 + |z|^{\ell} + |\bp|^{\ell}),\\
|D_z L(\bx,\bz,\bp)|&\le K_2(1 + |z|^{\ell-1}),\\
|D_p L(\bx, \bz, \bp)|&\le K_2(1 + |\bp|^{\ell-1}),
\end{aligned}
\end{equation}
for all $(\bx,\bz,\bp)\in\Omega\times \bbR\times \bbR^d$ and some constant $K_2>0$. From the results in~\cite{lcevans1998pde,gilbarg2015elliptic}, there exists a unique solution $\Psi\in \mathcal{W}$ satisfies
\begin{equation}\nonumber
I\left[\Psi\right] = \min_{w\in \mathcal{W}} I\left[w\right],
\end{equation}
and $\Psi \in  \mathcal{W}$ is the unique weak solution to~\eqref{eq:key}.
\end{proof}

Secondly, we prove a comparison result that is necessary to establish the non-negativity of the unique weak solution. 
Recall that $u \in H^1(\Omega)$ (\textit{resp.} $v \in H^1(\Omega)$) is called a weak \textit{supersolution} (\textit{resp.} \textit{subsolution}) to \eqref{eq:key} if 
\begin{align*}
 & \int_\Omega a \nabla u \cdot \nabla \phi \, d\bx \geq  - \int_\Omega f(\bx, u) \phi \, d\bx > -\infty \\
resp. \quad & \int_\Omega a \nabla v \cdot \nabla \phi \, d\bx \leq  - \int_\Omega f(\bx, v) \phi \, d\bx < \infty
\end{align*}
for any $\phi \in H^1_0(\Omega)$, $\phi \geq 0$ a.e.. The minus sign in front of $f$ comes from the negativity of the operator $\nabla \cdot a \nabla \cdot$.

\begin{lemma}\label{lm:2}
	Under the same hypotheses of Proposition~\ref{pr:1}, 
if $u$ and $v$ are weak supersolution and subsolution to \eqref{eq:key} respectively with $u\ge v$ on $\partial\Omega$ in the trace sense, then $u\ge v$ a.e. in $\Omega$.
\end{lemma}
%
\begin{proof}
	Subtracting the two inequalities in the definition of $u$ and $v$, we see that for any $\phi\in H^1_0(\Omega)$, $\phi\geq 0$ a.e. ,
	\begin{equation}\label{eq:comp2}
	\int_\Omega \left[ a\nabla(v-u) \cdot \nabla\phi + b(v-u)\phi \right] \,d\bx \leq  \int_\Omega c(|u|^{-(1+\theta)}u - |v|^{-(1+\theta)}v) \phi \, d\bx.
	\end{equation}
	In particular, pick $\phi=(v-u)^+$ where $(v-u)^+:=\max(v-u,0)$, then $\phi\in H^1_0(\Omega)$ since $u \geq v$ on $\partial\Omega$, $\phi\geq 0$ a.e., and
	$$
	\nabla\phi = \left\{ 
	\begin{array}{ll}
	\nabla(v-u) &  \text{ a.e. on } \{v\geq u\} \\
	0 &  \text{ a.e. on } \{v\leq u\} .
	\end{array}	
	\right.
	$$
	Hence
	$$
	\int_{\{v\geq u\}}  a|\nabla(v-u)|^2  + b|v-u|^2 \,d\bx \leq 0,
	$$
	which implies $u \geq v$ a.e. in $\Omega$, since $a$ is bounded from below by a positive constant and $b \geq 0$, 
\end{proof}

Now we are ready to prove Proposition \ref{pr:1}. Existence and uniqueness of the weak solution has been ensured by Lemma \ref{lm:1}. It remains to verify the desired non-negativity condition.

\begin{proof}[Proof of Proposition \ref{pr:1}]
	From Lemma~\ref{lm:1}, there is a unique weak solution $\Psi\in H^1(\Omega)$ to the equation~\eqref{eq:key}. On the other hand, $\Psi \equiv 0$ is also a solution to~\eqref{eq:key} with zero Dirichlet boundary condition. Lemma~\ref{lm:2} then implies $\Psi\ge 0$ a.e. in $\Omega$. Since we have
		\begin{equation}\nonumber
	-\nabla\cdot a \nabla \Psi =- f(\bx, \Psi) \le 0,~~~\text{ a.e. in }\Omega,
	\end{equation}
	then $\Psi$ is a subsolution to the linear equation $\nabla \cdot a \nabla u = 0$, and by the maximum principle $\sup_{\overline{\Omega}}\Psi = \sup_{\partial\Omega}\Psi = \sup_{\partial\Omega} g^{\frac{2}{1+\theta}} < \infty$. Without loss of generality, we assume $\Omega$ does not contain the origin so that there exist constants $l_0, l_1$ with $0<l_0 <|\bx| < l_1<\infty$ for all $\bx\in\Omega$,
	\begin{enumerate}
		\item 	when $-1 <\theta < 0$, there exists a subsolution to \eqref{eq:key} of the form $\underline{\Psi} = \kappa |\bx|^{2m}$. Indeed, being a subsolution means
		\begin{equation*}
		\kappa^{1+\theta} \Big(4am^2 + 2\big(\nabla a\cdot \bx +a(d-2)\big)m  -b|\bx|^2\Big) > c|\bx|^{2-2m(1+\theta)}.
		\end{equation*}
		Since $a$ and $|\nabla a|$ are bounded, we can always find $m>0, \kappa>0$ to make the inequality hold. Let $g_0 = \sup_{\partial\Omega}\underline{\Psi}$. If $g > g_0$, we conclude $\Psi \geq \underline{\Psi} > 0 $ a.e. by Lemma~\ref{lm:2}.
		\item when $\theta <-1$, let $\nu = (\sup_{\partial\Omega}g)^{-(1+\theta)} > 0$ and consider the equation
		\begin{equation*}
		\begin{aligned}
		-\nabla\cdot a\nabla \underline{\Psi} + (b + c\nu)\underline{\Psi} &= 0\quad\text{ in }\Omega ,\\
		\underline{\Psi} &= g\quad\text{ on }\partial\Omega.
		\end{aligned}
		\end{equation*}
		Since $\underline{\Psi} \le \sup_{\partial\Omega}g$ by the maximum principle, $\underline{\Psi}$ is actually a subsolution to \eqref{eq:key}, therefore $\Psi\ge \underline{\Psi} > 0$ due to strong maximum principle.
	\end{enumerate}

\end{proof}

\subsection{Case (2): $\theta \geq 0$}

When $\theta > 0$, the nonlinear term $f(\bx, z)$ in~\eqref{eq:key} has a singularity at $z = 0$. Before stating our result, we remark on two special values of $\theta$.

If $\theta = 0$, the equation~\eqref{eq:key} becomes the standard linear elliptic PDE $\nabla\cdot a \nabla \Psi = b\Psi + c$. Suppose $0$ is not an eigenvalue of the elliptic operator $-\nabla\cdot a\nabla + b$ with Dirichlet boundary condition, then there is a unique solution $\Psi$. 

If $\theta = 1$, then $\gamma_x=0$ from~\eqref{eq:theta}, and the internal data $Q$ in~\eqref{eq:hamilton1} has only the lower order term. In this circumstance, the equation~\eqref{eq:key} does not guarantee a unique solution when $c>0$, but two well chosen boundary conditions $(g_1,g_2)$ suffice to reconstruct $\sigma_{x,f}$ from the corresponding internal data $(Q_1,Q_2)$, see~\cite{bal2011non} for more details.

%
%
%
%
%

We therefore have to impose conditions on the signs of $b$ and $c$ in order to obtain a unique solution. Following the method in~\cite{crandall1977dirichlet}, we prove
\begin{proposition}\label{pr:3}
	If $\theta \geq 0$ and $b \geq 0$, $c \le 0$ a.e. in $\bx\in\Omega$,  
 then there is a unique positive weak solution $\Psi\in H^1(\Omega)$ to the equation~\eqref{eq:key}.
\end{proposition}
The proof of Proposition~\ref{pr:3} is based on the following lemma.
\begin{lemma}\label{lm:4}
	Under the same hypothesis of Proposition~\ref{pr:3}, let $\{\Psi_n\}_{n\in\bbN}$ satisfy the following equations
	\begin{equation}\label{eq:keyn}
	\begin{aligned}
	-\nabla\cdot a\nabla \Psi_n + b\Psi_n &= -c(\Psi_n + \frac{1}{n})^{-\theta},\quad&\text{ in }&\Omega\\
	\Psi_n &= g^{\frac{2}{1+\theta}},\quad&\text{ on }&\partial\Omega.
	\end{aligned}
	\end{equation}
Then we have
\begin{enumerate}
	\item There exists a unique positive weak solution $\Psi_n\in H^1(\Omega)$ to the equation~\eqref{eq:keyn}. 
	\item Each unique positive weak solution $\Psi_n$ is bounded from below and above, 
	\begin{equation}\nonumber
	v_n \ge \Psi_n(\bx) \ge \varphi(\bx) > 0.
	\end{equation} 
	where $\varphi(\bx)\in H^1(\Omega)$ is the weak solution to the linear elliptic equation
	\begin{equation}\label{eq:keyp}
	\begin{aligned}
	-\nabla\cdot a\nabla \varphi + b\varphi &= 0,\quad&\text{ in }&\Omega\\
	\varphi &= g^{\frac{2}{1+\theta}},\quad&\text{ on }&\partial\Omega
	\end{aligned}
	\end{equation}
	and $v_n\in H^1(\Omega)$ is the weak solution to the linear elliptic equation
	\begin{equation}\label{eq:keyv}
	\begin{aligned}
	-\nabla\cdot a\nabla v_n + bv_n &= -c(0 + \frac{1}{n})^{-\theta},\quad&\text{ in }&\Omega\\
	v_n &= g^{\frac{2}{1+\theta}},\quad&\text{ on }&\partial\Omega
	\end{aligned}
	\end{equation}
	\item For $n > m\ge 1$, we have
	\begin{equation}\nonumber
	\Psi_n \ge \Psi_m  \quad \text{ and } \quad  \Psi_n + \frac{1}{n} \le \Psi_m + \frac{1}{m} \quad\quad \text{ a.e. in } \Omega.
	\end{equation}
\end{enumerate}
\end{lemma}
\begin{proof}
First we prove the existence of positive weak solution to~\eqref{eq:keyn}. It is obvious that there exist a positive constant $\zeta > 0$ such that the solution $\varphi$ to~\eqref{eq:keyp} satisfies $\varphi > \zeta$ by strong maximum principle, therefore $\varphi$ is a weak subsolution to~\eqref{eq:keyn}. On the other hand, the solution $v_n$ to~\eqref{eq:keyv} is a weak supersolution to~\eqref{eq:keyn} and $v_n \ge \varphi$. Let $r_n(\bx, z) = -c(\bx) (z+n^{-1})^{-\theta}$, then $|\partial_z r_n(\bx, z)| = |c(\bx) \theta z^{-\theta - 1}|$ is bounded for $v_n\ge z\ge \varphi$ for all $\bx \in\overline{\Omega}$ and $|r_n|$ is dominated by function $|c(\bx) \varphi^{-\theta}|\in L^2(\Omega)$. Therefore by Perron's method~\cite[Chapter 9.3]{lcevans1998pde}, there exists a positive weak solution $\Psi_n\in H^1(\Omega)$ to~\eqref{eq:keyn} and $v_n \ge\Psi_n\ge \varphi > 0$.

As for the uniqueness, suppose there are two positive weak solutions $\Psi_n$ and $\widetilde{\Psi}_n$; subtract the weak forms of the equation \eqref{eq:keyn} to get that, for any $\phi\in H^1_0(\Omega)$, $\phi \geq 0$ a.e. ,
$$
\int_\Omega \left[ a \nabla (\Psi_n-\tilde{\Psi}_n) \cdot \nabla \phi + b (\Psi_n - \tilde{\Psi}_n) \phi \right] \, d\bx = \int_\Omega \left[ r_n(\bx, \Psi_n) - r_n(\bx, \tilde{\Psi}_n) \right] \phi \, d\bx. 
$$ 
In particular, take $\phi=(\Psi_n-\tilde{\Psi}_n)^+$, then $\phi \in H^1_0(\Omega)$ since $\Psi_n = \tilde{\Psi}_n$ on $\partial\Omega$. The right hand side is non-positive since $r_n(\bx, z)$ is non-increasing in $z$, that is,  
$r_n(\bx, \Psi_n) \leq r_n(\bx, \tilde{\Psi}_n)$ on $\{\Psi_n \geq \tilde{\Psi}_n\}$. Argue as in the proof of Lemma \ref{lm:2} to get
$$
\int_{\{\Psi_n \geq \tilde{\Psi}_n\}} a |\nabla (\Psi_n-\tilde{\Psi}_n)|^2  + b |\Psi_n - \tilde{\Psi}_n|^2  \, d\bx \leq 0,
$$ 
hence $\Psi_n \leq \tilde{\Psi}_n$ a.e. in $\Omega$. Switching the role of $\Psi_n$ and $\tilde{\Psi}_n$ yields the equality.

For the third part of the lemma, suppose $n > m \ge 1$. We apply the same technique for proving the uniqueness. Subtract the weak forms of the equations for $\Psi_n$ and $\Psi_m$ to get for any $\phi\in H^1_0(\Omega)$, $\phi \geq 0$ a.e. that,
\begin{equation} \label{eq:mn}
\int_\Omega \left[ a \nabla (\Psi_m - \Psi_n) \cdot \nabla \phi + b (\Psi_m - \Psi_n) \phi \right] \, d\bx = \int_\Omega \left[ r_m(\bx, \Psi_m) - r_n(\bx, \Psi_n) \right] \phi \, d\bx. 
\end{equation} 
Choose $\phi=(\Psi_m - \Psi_n)^+ \, \in H^1_0(\Omega)$ and observe that $r_m(\bx, \Psi_m) \leq r_n(\bx, \Psi_n)$ on $\{\Psi_m \geq \Psi_n\}$ since the function $x^{-\theta}$ is non-increasing for $x>0$. Repeat the argument in the uniqueness part to see that $\Psi_m \leq \Psi_n$ a.e. in $\Omega$.

Next, rewrite \eqref{eq:mn} as
\begin{align*}
 & \int_\Omega \left[ a \nabla \left( (\Psi_m+\frac{1}{m}) - (\Psi_n+\frac{1}{n}) \right) \cdot \nabla \phi + b \left((\Psi_m+\frac{1}{m}) - (\Psi_n+\frac{1}{n}) \right) \phi \right] \, d\bx \vspace{1ex} \\
 = & \int_\Omega b \, (\frac{1}{m} - \frac{1}{n}) \phi + \left[ r_m(\bx, \Psi_m) - r_n(\bx, \Psi_n) \right] \phi \, d\bx. 
\end{align*}
This time we choose $\phi=(\Psi_n + \frac{1}{n} - \Psi_m - \frac{1}{m})^+ \, \in H^1_0(\Omega)$. As $\Psi_n + \frac{1}{n} \geq \Psi_m + \frac{1}{m}$ implies $\Psi_n \geq \Psi_m$ hence $r_n(x,\Psi_n) \leq r_m(x,\Psi_m)$, the right hand side is non-negative on $\{\Psi_n + \frac{1}{n} \geq \Psi_m + \frac{1}{m}\}$. We obtain
$$
- \int_{\{\Psi_n + \frac{1}{n} \geq \Psi_m + \frac{1}{m}\}}  a  | \ \nabla (\Psi_m+\frac{1}{m}) - \nabla (\Psi_n+\frac{1}{n}) |^2  + b |(\Psi_m+\frac{1}{m}) - (\Psi_n+\frac{1}{n}) |^2 \, d\bx \geq 0,
$$
showing $\Psi_n + \frac{1}{n} \leq \Psi_m + \frac{1}{m}$ a.e. in $\Omega$.

%
\end{proof}
Now, we are ready to prove Proposition~\ref{pr:3}.
\begin{proof}[Proof of Proposition \ref{pr:3}]
From Lemma \ref{lm:4}, we can construct a sequence $\{\Psi_n\}_{n\in\mathbb{N}}$ where $\Psi_n$ is the unique positive weak solution to~\eqref{eq:keyn}. Take $\Psi_n$ and $\Psi_m$ from this sequence with $n>m\geq 1$, then $0\leq \Psi_n - \Psi_m \leq \frac{1}{m} - \frac{1}{n}$ a.e., hence $\Psi_n - \Psi_m$ is a Cauchy sequence in $L^2(\Omega)$.
Moreover, $\Psi_n - \Psi_m$ satisfies 
\begin{align*}
-\nabla \cdot a \nabla (\Psi_n - \Psi_m) + b (\Psi_n - \Psi_m) & = -c(\Psi_n + \frac{1}{n})^{-\theta} + c(\Psi_m + \frac{1}{m})^{-\theta} & \text{ in } \Omega \\
\Psi_n - \Psi_m & = 0 & \text{ on } \partial\Omega.
\end{align*}
By the $H^1$ bound for the solution
\begin{align*}
\|\Psi_n - \Psi_m\|_{H^1(\Omega)} & \leq K \| -c(\Psi_n + \frac{1}{n})^{-\theta} + c(\Psi_m + \frac{1}{m})^{-\theta} \|_{L^2(\Omega)} \\
 & \leq K \|c\|_{L^\infty(\Omega)} \| \Psi_n - \Psi_m + \frac{1}{n} - \frac{1}{m} \|_{L^2(\Omega)} \\
 & \leq K \|c\|_{L^\infty(\Omega)} ( \| \Psi_n - \Psi_m \|_{L^2(\Omega)} + \| \frac{1}{n} - \frac{1}{m} \|_{L^2(\Omega)} )
\end{align*}
we see that $\Psi_n - \Psi_m$ is a Cauchy sequence in $H^1(\Omega)$ as well. Here $K=K(d,a,b,\partial\Omega)>0$ is a constant, and the second inequality is valid as the function $x^{-\theta}$ is Lipschitz continuous when $x>0$ is bounded away from $0$. Then there exists $\Psi\in H^1(\Omega)$ such that $\Psi_n \rightarrow \Psi$ in $H^1(\Omega)$ and $(c\Psi_n+\frac{1}{n})^{-\theta} \rightarrow c\Psi^{-\theta}$ in $L^2(\Omega)$ as $n\rightarrow\infty$. Taking the limit in \eqref{eq:keyn} implies $\Psi$ is a positive weak solution to~\eqref{eq:key}.
As $\Psi_n$ is increasing with respect to $n$, we have $\Psi \geq \Psi_n \geq \varphi >0$ where $\varphi$ is the subsolution in Lemma \ref{lm:4}. The uniqueness part can be proved as in Lemma \ref{lm:4}.
\end{proof}

\section{Reconstruction of $\sigma_{x,f}$: stability}\label{SEC:REC1_2}


So far, we have established in Proposition \ref{pr:1} and Proposition \ref{pr:3} the existence and uniqueness of a positive weak solution to \eqref{eq:key} in the following two cases:
\begin{itemize}
\item[(1)] $-1 \neq \theta < 0$, $b \geq 0$ and $c \geq 0$ a.e. in $\Omega$;
\item[(2)] $\theta \geq 0$, $b \geq 0$ and $c \leq 0$ a.e. in $\Omega$.
\end{itemize}
In Case (1), the additional boundary condition $g>g_0$ on $\partial\Omega$ is required for some $g_0>0$.  We show stable dependence of the solution on the coefficient $c$ in this section. The above two cases can be simultaneously treated for this purpose.

\begin{lemma}\label{lm:5}
In either Case (1) or Case (2), suppose $\Psi_1$ and $\Psi_2$ are the unique positive weak solutions to the equation~\eqref{eq:key} with nonlinear term's coefficient as $c_1$ and $c_2$ respectively. Then there exists a constant $K_5 > 0$ such that
	\begin{equation}\nonumber
	\|\Psi_1 - \Psi_2\|_{H^{1}(\Omega)} \le K_5 \|c_1 - c_2\|_{L^{2}(\Omega)}.
	\end{equation}
\end{lemma}
\begin{proof}
	Given $\theta \neq -1$, there exist constants $K_3, K_4 > 0$ such that the inequality $K_4 \le \Psi_1^{-\theta}, \Psi_2^{-\theta} \le K_3$ holds a.e. in $\Omega$ by Proposition~\ref{pr:1},  Proposition~\ref{pr:3}, and the standard local estimate for elliptic PDEs~\cite{gilbarg2015elliptic}. The difference $\omega := \Psi_1 - \Psi_2$ solves
	\begin{equation}\nonumber
	\begin{aligned}
	-\nabla \cdot a \nabla \omega + b \omega &= -(c_1-c_2)\Psi_1^{-\theta} - c_2 (\Psi_1^{-\theta} - \Psi_2^{-\theta}),\quad &\text{ in }&\Omega\\
	\omega &= 0.\quad &\text{ on }&\partial\Omega
	\end{aligned}
	\end{equation}
	Multiply the above equation by $\omega$ and integrate over $\Omega$:
		\begin{equation}\nonumber
		\int_{\Omega} a|\nabla \omega|^2 + b|\omega|^2 + c_2(\Psi_1^{-\theta} - \Psi_2^{-\theta})\omega \, d\bx = \int_{\Omega} -(c_1 - c_2)\Psi_1^{-\theta}\omega \, d\bx.
		\end{equation}
	Since $c_2(\Psi_1^{-\theta} - \Psi_2^{-\theta})\omega \ge 0$ in either case, there exists a constant $K_5 > 0$ so that
		\begin{equation}
		\|\omega\|_{H^1(\Omega)}\le K_5 \|c_1-c_2\|_{L^2(\Omega)}.
		\end{equation}
\end{proof}

\begin{remark}
It is easy to see the positive solutions $\Psi_1$ and $\Psi_2$ also satisfy 
\begin{equation}\label{eq:psi}
\|\Psi_1^{\frac{1+\theta}{2}} - \Psi_2^{\frac{1+\theta}{2}}\|_{H^1(\Omega)} \le K_{7}\|\Psi_1 - \Psi_2\|_{H^1(\Omega)}
\end{equation}
for some constant $K_7 > 0$.
\end{remark}

Now we prove a stability estimate on the reconstruction of $\sigma_{x,f}$ with the help of the above lemma. 

\begin{theorem}\label{thm:1}
In either Case (1) or Case (2), suppose $Q_1$ and $Q_2$ are the internal data with the absorption coefficients $\sigma_{x,f,1}$ and $\sigma_{x,f,2}$ respectively. Then 
there exists a constant $K_{14} > 0$ such that
\begin{equation}\nonumber
\|\sigma_{x,f,1} - \sigma_{x,f,2}\|_{L^1(\Omega)} \le K_{14}\left(\|Q_1 - Q_2\|_{L^1(\Omega)} + \|Q_1 - Q_2\|^2_{L^2(\Omega)}\right).
\end{equation}
\end{theorem}
\begin{proof}
	Given a positive solution $\Psi$ to~\eqref{eq:key}, then $u_0 = \Psi^{\frac{1+\theta}{2}}$ (see~\eqref{eq:theta}), and 
$\sigma_{x,f}$ can be reconstructed through (see~\eqref{eq:hamilton1})
	\begin{equation} \label{eq:sigmaxf}
	\sigma_{x,f} = \frac{Q - \gamma_x D_x |\nabla u_0|^2 - \beta_x\sigma_{x,a}|u_0|^2}{\beta_f |u_0|^2}.
	\end{equation}

	Since both $\Psi_1$ and $\Psi_2$ are bounded from below and above by positive constants, as explained in the proof of Lemma \ref{lm:5}, one can find constants $K_9, K_{10}, K_{11} >0$ such that 
	\begin{equation}\nonumber
		K_9 > u_{0,1}, u_{0,2} > K_{10}, \text{ and } K_{11} > \|u_{0,1}\|_{H^1(\Omega)},\|u_{0,2}\|_{H^1(\Omega)}.
	\end{equation}
Then
	\begin{equation}\nonumber
	\begin{aligned}
	|\sigma_{x,f,2} - \sigma_{x,f,1}| &= \Big|\frac{1}{\beta_f}\frac{Q_2 |u_{0,1}|^2 - Q_1 |u_{0,2}|^2}{|u_{0,1}|^2 |u_{0,2}|^2} -\frac{\gamma_x D_x}{\beta_f}\frac{|\nabla u_{0,2}|^2 |u_{0,1}|^2 - |\nabla u_{0,1}|^2 |u_{0,2}|^2}{|u_{0,1}|^2 |u_{0,2}|^2}\Big|\\
	&\le K_{12}\left(|Q_1-Q_2| + |u_{0,1} - u_{0,2}|^2 + |\nabla (u_{0.1}  - u_{0,2})|^2\right)	
	\end{aligned}
	\end{equation}
for some constant $K_{12} > 0$. Integrating over $\Omega$ and making use of Lemma \ref{lm:5} and \eqref{eq:psi}, we conclude
	\begin{equation}\nonumber
	\begin{aligned}
	\|\sigma_{x,f,2} - \sigma_{x,f,1}\|_{L^1(\Omega)} &\le K_{12}\left(\|Q_1 - Q_2\|_{L^1(\Omega)} + \|u_{0,1} - u_{0,2}\|^2_{H^1(\Omega)}\right)\\
	&\le K_{13}\left(\|Q_1 - Q_2\|_{L^1(\Omega)} + \|c_1 - c_2\|^2_{L^2(\Omega)}\right) \\
	&\le K_{14}\left(\|Q_1 - Q_2\|_{L^1(\Omega)} + \|Q_1 - Q_2\|^2_{L^2(\Omega)}\right)
	\end{aligned}
	\end{equation}
	for some constants $K_{13}, K_{14} > 0$.

\end{proof}
\section{Reconstruction of $\sigma_{x,f}$: algorithms}\label{SEC:NUM}
We develop some algorithms\footnote{The algorithms and numerical experiments are all implemented in MATLAB, the code repository is at \href{https://github.com/lowrank/fumot}{https://github.com/lowrank/fumot}.} in this section to reconstruct $\sigma_{x,f}$. 
The key step is to find $\Psi$ in~\eqref{eq:key}, then $u_0 = \Psi^{\frac{1+\theta}{2}}$ and $\sigma_{x,f}$ can be reconstructed from \eqref{eq:sigmaxf}. The equation~\eqref{eq:key} is semilinear and one can solve it using iterative methods such as Newton's method. However, Newton-type methods generally converge fast but only have local convergence, unless other properties such as monotonicity or convexity are available~\cite{han1977globally,hiebert1982evaluation,ralph1994global}.  In order to guarantee the convergence, we propose the following simple but effective iterative schemes for~\eqref{eq:key} in the following three scenarios. Note that Case (1) in Section \ref{SEC:REC1} are divided into two sub-cases (1.1) and (1.2) here.
\begin{itemize}
	\item[(1.1)] $-1 < \theta < 0$, $b \geq 0$ and $c \ge 0$; 
	\item[(1.2)] $\theta < -1$, $b \geq 0$ and $c \ge 0$; 
	\item[(2)] $\theta \geq 0$, $b \geq 0$ and $c \leq 0$. 
\end{itemize}


\subsection{Case (1.1): $-1 < \theta < 0$, $b \geq 0$ and $c \ge 0$}

We propose the following iterative algorithm for this case.

\begin{algorithm}[hbt!]
	\caption{iterative method for $\Psi$ when $-1 < \theta < 0$, $b \geq 0$ and $c \ge 0$}\label{alg:iter1}
	\begin{algorithmic}[1]
		\STATE $r \gets 1$ 
		\STATE $n \gets 1$
	    \STATE Solve  $-\nabla \cdot a\nabla \Psi_0 + b \Psi_0 = 0$ with boundary condition $g^{\frac{2}{1+\theta}}$
		\WHILE{$r > \varepsilon$} 
		\STATE Solve $-\nabla\cdot a\nabla \Psi_n + b \Psi_n + c \Psi_{n-1}^{-(\theta + 1)} \Psi_n = 0$ with boundary condition $g^{\frac{2}{1+\theta}}$
		\STATE $r \gets \|\Psi_n - \Psi_{n-1}\|_{L^2(\Omega)}/\|\Psi_n\|_{L^2(\Omega)}$
		\STATE $n\gets n+1$
		\ENDWHILE
		\STATE \textbf{return} $\Psi_n$
	\end{algorithmic}
\end{algorithm}
\begin{theorem}\label{thm:3}
	If $-1 < \theta < 0$, $b \geq 0$ and $c \ge 0$, the sequence $\{\Psi_n\}_{n\ge 0}$ in Algorithm~\ref{alg:iter1} converges to the unique positive weak solution $\Psi$ of \eqref{eq:key}.
\end{theorem}
\begin{proof}
	The proof is an inductive process. The inequalities below should be interpreted in the weak sense by testing them on $\phi \in H^1_0(\Omega)$ with $\phi \geq 0$.
	
	First, we show the sequence $\{\Psi_n\}_{n\ge 0}$ is bounded from below by $\Psi$. For the base step, $\Psi_0$ satisfies
		\begin{equation}\nonumber
		-\nabla\cdot a\nabla \Psi_0 + b\Psi_0 + c\Psi_0^{-\theta} = c\Psi_0^{-\theta} \ge 0
		\end{equation}
	and Lemma \ref{lm:1} implies $\Psi_0 \ge \Psi$ a.e. in $\Omega$. For the inductive step, fix an integer $n$ and suppose we have proved $\Psi_m \ge \Psi$ for any $m<n$, then
		\begin{equation}\nonumber
		-\nabla \cdot a\nabla (\Psi_n-\Psi) + b(\Psi_n-\Psi) + c(\Psi_{n-1}^{-(\theta + 1)}\Psi_n - \Psi^{-(\theta+1)}\Psi) = 0
		\end{equation}\nonumber
Since $-(\theta + 1) < 0$, we have $\Psi_{n-1}^{-(\theta + 1)} \le \Psi^{-(\theta + 1)}$ hence 		
		\begin{equation}\nonumber
		-\nabla \cdot a\nabla (\Psi_n-\Psi) + b(\Psi_n-\Psi) + c\Psi_{n-1}^{-(\theta + 1)}(\Psi_n - \Psi) \ge 0.
		\end{equation}
Lemma \ref{lm:1} implies $\Psi_n \ge \Psi$ a.e. in $\Omega$. This completes the induction.

	Next, we show the sequence $\{\Psi_n\}_{n\ge 0}$ is decreasing. For the base step, $\Psi_0 - \Psi_1$ satisfies
		\begin{equation}\nonumber
		-\nabla \cdot a\nabla (\Psi_0 - \Psi_1) + b(\Psi_0-\Psi_1) = c\Psi_0^{-(\theta + 1)}\Psi_1 \ge 0,
		\end{equation}
		hence $\Psi_0 \ge \Psi_1$ by Lemma \ref{lm:1}. For the inductive step, fix an integer $n$ and suppose we have proved $\Psi_{m-1}\ge \Psi_m$ for any $m<n$. Subtracting the equations for $\Psi_n$ and $\Psi_{n-1}$ we have
		\begin{equation}\nonumber
		-\nabla \cdot a\nabla (\Psi_{n}-\Psi_{n-1}) + b(\Psi_n-\Psi_{n-1}) + c(\Psi_{n-1}^{-(\theta + 1)}\Psi_{n} - \Psi_{n-2}^{-(\theta+1)}\Psi_{n-1}) = 0.
		\end{equation}
		As $\Psi_{n-2}\ge \Psi_{n-1}$ by the inductive assumption, we obtain
		\begin{equation}\nonumber
		-\nabla \cdot a\nabla (\Psi_{n}-\Psi_{n-1}) + b(\Psi_n-\Psi_{n-1}) + c\Psi_{n-1}^{-(\theta + 1)}(\Psi_{n} - \Psi_{n-1}) \le 0,
		\end{equation}
		which implies $\Psi_n \le \Psi_{n-1}$. The induction is complete.
		
		We have showed the sequence $\{\Psi_n\}_{n\ge 0}$ is decreasing and bounded from below by $\Psi$, therefore $\{\Psi_n\}_{n\ge 0}$ converges uniformly to a limit, say $\tilde{\Psi}$, a.e. in $\Omega$. The convergence holds in $L^2(\Omega)$ by the dominant convergence theorem, and in $H^1(\Omega)$ by the elliptic regularity estimate. This forces $\tilde{\Psi} = \Psi$.
\end{proof}
\subsection{Case (1.2): $\theta < -1$, $b \geq 0$ and $c \ge 0$}
For this case, we propose the following algorithm.
\begin{algorithm}[hbt!]
	\caption{iterative method for $\Psi$ when $\theta < -1$, $b \geq 0$ and $c \ge 0$}\label{alg:iter2}
	\begin{algorithmic}[1]
		\STATE $r \gets 1$ 
		\STATE $n \gets 1$
		\STATE $\overline{h} \gets \max_{\bx\in\partial\Omega} g(\bx)^{\frac{2}{1+\theta}}$
		\STATE $\nu \gets -\theta \overline{h}^{-(1+\theta)}$
		\STATE Solve $-\nabla\cdot a \nabla\Psi_0 + (b + c \nu)\Psi_0 = 0$ with boundary condition $g^{\frac{2}{1+\theta}}$
		\WHILE{$r > \varepsilon$} 
		\STATE Solve $-\nabla\cdot a\nabla \Psi_n + (b+c \nu) \Psi_n =- c( \Psi_{n-1}^{-\theta} -\nu \Psi_{n-1})$ with boundary condition $g^{\frac{2}{1+\theta}}$
		\STATE $r \gets \|\Psi_n - \Psi_{n-1}\|_{L^2(\Omega)}/\|\Psi_n\|_{L^2(\Omega)}$
		\STATE $n\gets n+1$
		\ENDWHILE
		\STATE \textbf{return} $\Psi_n$
	\end{algorithmic}
\end{algorithm}
\begin{theorem}\label{thm:4}
	If $\theta < -1$, $b \geq 0$ and $c \ge 0$, the sequence $\{\Psi_n\}_{n\ge 0}$ in Algorithm~\ref{alg:iter2} converges to the unique positive weak solution $\Psi$ of \eqref{eq:key}.
\end{theorem}
\begin{proof}
The assumption implies $-(1+\theta) > 0$. Recall that $\Psi$ satisfies the maximum principle $\Psi\le\overline{h}= \max_{\bx\in\partial\Omega} g(\bx)^{\frac{2}{1+\theta}}$ according to Proposition~\ref{pr:1}.

We prove $\Psi_n \leq \Psi$ a.e. by induction. For the base step,  the initial solution $\Psi_0$ satisfies
\begin{equation}
-\nabla \cdot a\nabla\Psi_0 + (b+c\Psi_0^{-(\theta + 1)}) \Psi_0 = c(\Psi_0^{-(\theta + 1)} - \nu)\Psi_0 \le 0,
\end{equation}
thus $\Psi_0 \le\Psi$ a.e. in $\Omega$ by Lemma \ref{lm:1}. For the inductive step, fix an integer $n$ and suppose we have proved $\Psi_m \le \Psi$ for all $m<n$, then
\begin{equation}
-\nabla\cdot a\nabla(\Psi - \Psi_n) + (b+c\nu)(\Psi - \Psi_n) = c\left((\nu\Psi - \Psi^{-\theta}) - (\nu\Psi_{n-1} - \Psi_{n-1}^{-\theta})\right)\ge 0.
\end{equation}
The choice of $\nu$ makes $f(z) = \nu z - z^{-\theta}$ an increasing function for $0\le z\le \overline{h}$, one therefore has $\Psi \ge \Psi_n$ by the maximum principle, completing the induction.

Next, we show the sequence $\{\Psi_n\}_{n\ge 0}$ is increasing. For the base step, $\Psi_1$ satisfies
\begin{equation}
-\nabla\cdot a\nabla (\Psi_1 - \Psi_0) + (b + c\nu)(\Psi_1 - \Psi_0) = c(\nu \Psi_1 - \Psi_1^{-\theta})\ge 0,
\end{equation}
hence $\Psi_1\ge \Psi_0$ a.e. by the maximum principle. For the inductive step, simply notice the assumption $\Psi_{n-1}\ge \Psi_{n-2}$ and
\small
\begin{equation}
-\nabla \cdot a\nabla (\Psi_n - \Psi_{n-1}) + (b + c\nu)(\Psi_n - \Psi_{n-1}) = c\left((\nu \Psi_{n-1} - \Psi_{n-1}^{-\theta}) - (\nu\Psi_{n-2} - \Psi_{n-2}^{-\theta})\right)\ge 0,
\end{equation}
\normalsize
implies $\Psi_{n} \ge \Psi_{n-1}$ a.e. in $\Omega$. A similar argument as in the proof of Theorem~\ref{thm:3} verifies the increasing and bounded sequence $\{\Psi_n\}_{n\ge 0}$ actually converges to $\Psi$.
\end{proof}
\subsection{Case (2): $\theta \geq 0$, $b \geq 0$ and $c \leq 0$}
We propose the following algorithm for this case.
\begin{algorithm}[hbt!]
	\caption{iterative method for $\Psi$ when $\theta \geq 0$, $b \geq 0$ and $c \leq 0$}\label{alg:iter3}
	\begin{algorithmic}[1]
		\STATE $r \gets 1$ 
		\STATE $n \gets 1$
		\STATE Solve $-\nabla\cdot a\nabla\Psi_0 + b\Psi_0 = 0$ with boundary condition $g^{\frac{2}{1+\theta}}$
		\STATE $\kappa \gets \min_{\bx\in\Omega} g(\bx)^{\frac{2}{1+\theta}}$
		\STATE $\nu \gets -\theta \kappa^{-(1+\theta)}$
		\WHILE{$r > \varepsilon$} 
		\STATE Solve $-\nabla\cdot a\nabla \Psi_n + (b +c \nu) \Psi_n =- c( \Psi_{n-1}^{-\theta} -\nu \Psi_{n-1})$ with boundary condition $g^{\frac{2}{1+\theta}}$
		\STATE $r \gets \|\Psi_n - \Psi_{n-1}\|_{L^2(\Omega)}/\|\Psi_n\|_{L^2(\Omega)}$
		\STATE $\kappa\gets \min_{\bx\in\Omega} \Psi_n(\bx)$ 
		\STATE $\nu \gets -\theta \kappa^{-(1+\theta)}$
		\STATE $n\gets n+1$
		\ENDWHILE
		\STATE \textbf{return} $\Psi^n$
	\end{algorithmic}
\end{algorithm}
\begin{theorem}\label{thm:5}
	If $\theta \geq 0$, $b \geq 0$ and $c \leq 0$, the sequence $\{\Psi_n\}_{n\ge 0}$ in Algorithm~\ref{alg:iter3} converges to the unique positive weak solution $\Psi$ of \eqref{eq:key}.
\end{theorem}
\begin{proof}
	 Recall that we showed $0 < \kappa \leq \Psi_0(\bx) \leq \Psi$ in Proposition~\ref{pr:3}, where $\Psi_0$ was named $\varphi$ at that time. We prove inductively that the sequence $\{\Psi_n\}_{n\ge 0}$ is again increasing and bounded from above by $\Psi$. The convergence then follows the same way as in the proof of Theorem \ref{thm:3}.

For the boundedness, the base step has been verified. To establish the inductive step, suppose $\kappa \leq \Psi_{n-1} \leq \Psi$, then
	\begin{equation}
	-\nabla\cdot a\nabla(\Psi - \Psi_m) + (b+c\nu)(\Psi - \Psi_m) = -c\left((\Psi^{-\theta} - \nu\Psi) - ( \Psi_{n-1}^{-\theta}-\nu\Psi_{n-1} )\right)\ge 0,
	\end{equation}
since the function $f(z) = z^{-\theta} - \nu z $ is increasing for $z \geq \kappa$. Hence $\kappa \leq \Psi_{n} \leq \Psi$ by the maximum principle.

For the monotonicity, the base step $\Psi_1 \geq \Psi_0$ follows from 
	\begin{equation}
	-\nabla\cdot a\nabla (\Psi_1 - \Psi_0) + (b + c\nu)(\Psi_1 - \Psi_0) = -c\Psi_0^{-\theta} \ge 0
	\end{equation}
and the maximum principle. The inductive step follows from the assumption $\Psi_{n-1} \geq \Psi_{n-2}$ and
\small
	\begin{equation}
	-\nabla \cdot a\nabla (\Psi_n - \Psi_{n-1}) + (b + c\nu)(\Psi_n - \Psi_{n-1}) = -c\left(( \Psi_{n-1}^{-\theta}-\nu \Psi_{n-1}) - ( \Psi_{n-2}^{-\theta}-\nu\Psi_{n-2})\right)\ge 0.
	\end{equation}
\normalsize
together with the maximum principle.
%
\end{proof}

\section{Reconstruction of $\eta$}\label{SEC:REC2}

We turn to the reconstruction of the quantum efficiency coefficient $\eta$ from the second internal data $S$, assuming that $\sigma_{x,f}$ has been successfully recovered based on the discussion in the previous sections. We will closely follow the treatment in~\cite{bal2014ultrasound}.

Recall that $w_0$ solves \eqref{eq:fl2w} with $\epsilon = 0$; $v$ and $p$ are the auxiliary functions obeying~\eqref{eq:v} and \eqref{eq:p}. Introduce the operators $\mathcal{T}_0, \mathcal{T}_1$ by
\begin{equation}\nonumber
\begin{aligned}
\mathcal{T}_0 \eta &:= w_0 = (-\nabla \cdot D_m\nabla + \sigma_{m,a})^{-1} (\eta \sigma_{x,f} u_0),\\
\mathcal{T}_1 \eta &:= p = (-\nabla\cdot D_x\nabla + (\sigma_{x,a} + \sigma_{x,f}))^{-1}(\eta \sigma_{x,f} v)
\end{aligned}
\end{equation}
subject to zero Dirichlet boundary conditions. Since the boundary $\partial\Omega$ is $C^2$, the operators $\mathcal{T}_0, \mathcal{T}_1:L^2(\Omega)\to H^2(\Omega)$ are bounded~\cite[Theorem 8.12]{gilbarg2015elliptic}. Then the second internal data~\eqref{eq:S} can be written as
\begin{align*}
S & = \gamma_m D_m  \nabla v \cdot \nabla (\mathcal{T}_0 \eta) + \beta_m \sigma_{m,a} v (\mathcal{T}_0 \eta) + \gamma_x D_x \nabla u_0 \cdot \nabla (\mathcal{T}_1 \eta) \\
 & \quad\quad + (\beta_x \sigma_{x,a} + \beta_f \sigma_{x,f}) u_0 (\mathcal{T}_1 \eta) - \eta \beta_f \sigma_{x,f} u_0 v \\
& := \mathcal{A}_1 + \mathcal{A}_2 + \mathcal{A}_0 
\end{align*}
where the linear operators $\mathcal{A}_0, \mathcal{A}_1,\mathcal{A}_2$ are defined by
\begin{equation}\nonumber
\begin{aligned}
\mathcal{A}_1 \eta &:= (\gamma_m D_m  \nabla v\cdot \nabla + \beta_m \sigma_{m,a} v)\mathcal{T}_0 \eta,\\
\mathcal{A}_2 \eta &:= \left(\gamma_x D_x \nabla u_0\cdot \nabla + (\beta_x\sigma_{x,a} + \beta_f\sigma_{x,f})u_0\right) \mathcal{T}_1\eta \\
\mathcal{A}_0 \eta &:= -\eta \beta_f \sigma_{x,f}u_0 v.
\end{aligned}
\end{equation}
We therefore have the following identity
\begin{equation}\label{eq:eta}
(\mathcal{A}_0 +\mathcal{A}_1 +\mathcal{A}_2)\eta = S.
\end{equation}
Uniqueness and stability follows readily from here, and $\eta$ can be explicitly reconstructed simply by solving this linear equation.

\begin{theorem}
Suppose $0$ is not an eigenvalue of $(\mathcal{A}_0 + \mathcal{A}_1 + \mathcal{A}_2)$, then $\eta$ is uniquely determined by the internal data $S$. Moreover, if $S_1$ and $S_2$ are two sets of internal data with quantum efficiency coefficients $\eta_1$ and $\eta_2$ respectively, then there exists a constant $K_{15} > 0$ such that
	\begin{equation}
	\|\eta_1 - \eta_2\|_{L^2(\Omega)}\le K_{15}\|S_1 - S_2\|_{L^2(\Omega)} .
	\end{equation}
\end{theorem}
\begin{proof}
As $\mathcal{A}_1,\mathcal{A}_2: L^2(\Omega) \cap L^{\infty}(\Omega)\to H^1(\Omega)$ are bounded and the inclusion $H^1(\Omega)\subset L^2(\Omega)$ is compact by the Sobolev embedding theorem, $\mathcal{A}_1, \mathcal{A}_2: L^2(\Omega) \cap L^{\infty}(\Omega)\to L^2(\Omega)$ are compact operators, which implies the sum $\mathcal{A}_0 + \mathcal{A}_1 + \mathcal{A}_2:L^2(\Omega)\to L^2(\Omega)$ is a Fredholm operator. If $0$ is not an eigenvalue, then $(\mathcal{A}_0 + \mathcal{A}_1 + \mathcal{A}_2)$ has a bounded inverse, hence the solution $\eta$ is unique. The constant $K_{15}$ in the stability estimate can be taken as the operator norm of the inverse, which is a bounded linear operator on $L^2(\Omega)$.
\end{proof}

\section{Numerical experiments}\label{SEC:EXP}

In this section, we numerically implement the proposed reconstruction procedures for $\sigma_{x,f}$ in Section \ref{SEC:NUM} and for $\eta$ in Section \ref{SEC:REC2}.
We perform the reconstructions in 2D domain $[-0.5, 0.5]^2$ which is triangulated into $37008$ triangles. The $4$-th order Lagrange finite element method is employed to solve the equations. Some of background coefficients that we assume to be known are set as follows
\begin{equation}\nonumber
\begin{aligned}
&D_x \equiv 0.1,\quad &D_m =0.1 + 0.02\cos(2x)\cos(2y), \\
& \sigma_{x,a} \equiv 0.1,\quad &\sigma_{m,a} = 0.1 + 0.02\cos(4x^2 + 4y^2).
\end{aligned}
\end{equation}
The absorption coefficient $\sigma_{x,f}$ and quantum efficiency $\eta$ are shown in Figure~\ref{fig:1}.
\begin{figure}[!htb]
	\centering
	\includegraphics[scale=0.35]{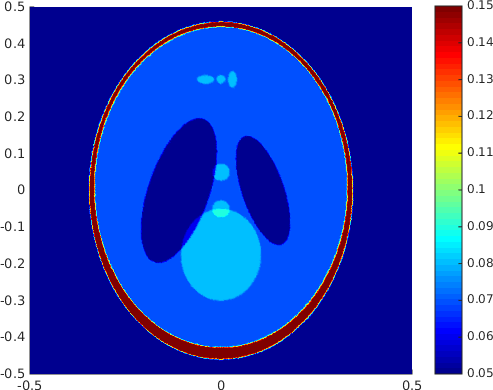}\quad
	\includegraphics[scale=0.35]{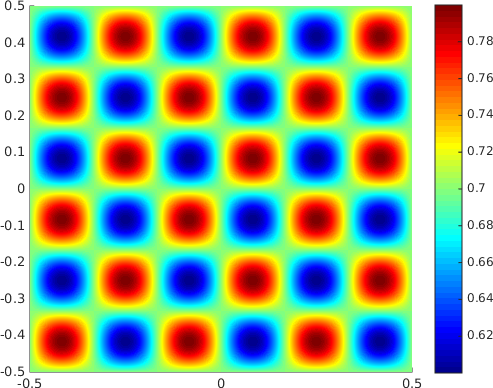}
	\caption{Left: The absorption coefficient $\sigma_{x,f}$ of fluorophores. Right: The quantum efficiency coefficient $\eta$.}
	\label{fig:1}
\end{figure}
The excitation source on the boundary is the restriction of the $C^2$ function $$g(x,y) = e^{2x} + e^{-2y}.$$

We demonstrate three examples based on the signs of $\theta$ ,$b$ and $c$ as discussed in Section \ref{SEC:NUM}.  To reconstruct $\sigma_{x,f}$, we implement the corresponding iterative algorithms.
To reconstruct $\eta$, we discretize $\mathcal{A}_0, \mathcal{A}_1,\mathcal{A}_2$ in \eqref{eq:eta} using Lagrangian finite element method of order $k\ge 3$, and then solve the resulting linear system with the Krylov subspace method (restarted GMRES). It should be noted that such finite element solution only belongs to $C^{0,\alpha}(\Omega)$, thus the gradients that appear in $\mathcal{A}_1, \mathcal{A}_2$ may not agree on the interface of adjacent elements. Such difference however could be negligible when the mesh is sufficiently fine and the exact solution is regular enough according to the $L^{\infty}$ estimate in~\cite{scott1976optimal}. It indicates that the error of gradients on each element is bounded by
\begin{equation}
\max_{T_l\in T} \|\nabla u - \nabla u^{\ast}\|_{\infty} \le O( h^{k-1} |u|_{W_{\infty}^k(\Omega)}),
\end{equation} 
where $u^{\ast}$ is the finite element solution and $u$ is the exact weak solution on mesh $T=\{T_l\}_{l=1}^N$.

%
%
%

\medskip
\textbf{Case (1.1): $-1 < \theta < 0$, $b \geq 0$ and $c \ge 0$.} In this example, the intrinsic elasto-optical parameters are $n_x = -0.8$, $n_m = -0.7$, $n_f = -0.9$, then
\begin{equation}\nonumber
\begin{aligned}
&\gamma_x = -2.6,\quad \gamma_m = -2.4, \\
&\beta_x = -0.6,\quad \beta_m = -0.4,\quad \beta_f = -0.8,
\end{aligned}
\end{equation}
and $\tau = 3.25$ and $\mu = -0.25$, $\theta = -\frac{9}{17}$.  The coefficients in~\eqref{eq:key} satisfy $b \geq 0$ and $c \ge 0$, hence 
we can apply Algorithm~\ref{alg:iter1} to solve the nonlinear equation. The reconstructed images of $\sigma_{x,f}$ and $\eta$, with and without noises, are illustrated in Figure~\ref{fig:1a}. 
\begin{figure}[!htb]
	\centering
	\includegraphics[width=0.3\textwidth,height=0.16\textheight]{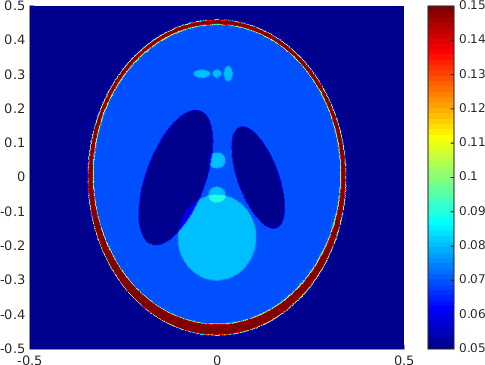}
	\includegraphics[width=0.3\textwidth,height=0.16\textheight]{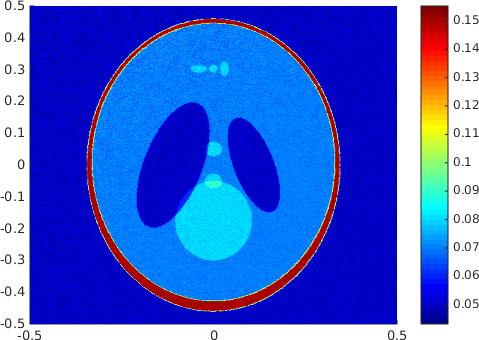}
	\includegraphics[width=0.3\textwidth,height=0.16\textheight]{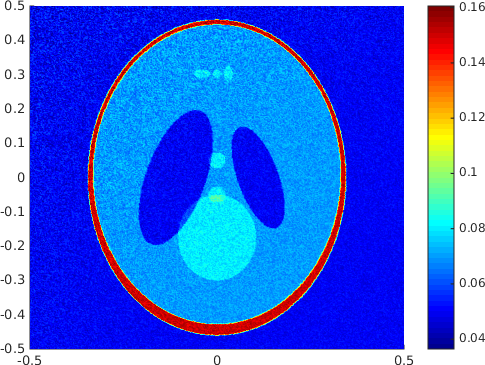}\\
	\includegraphics[width=0.3\textwidth,height=0.16\textheight]{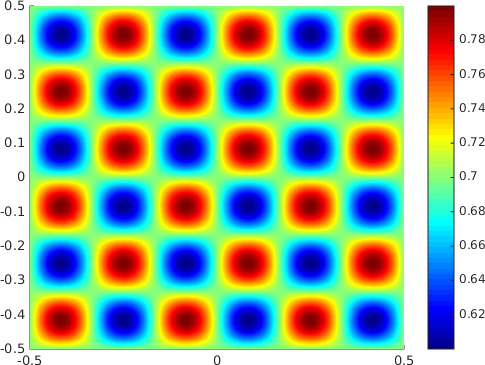}
	\includegraphics[width=0.3\textwidth,height=0.16\textheight]{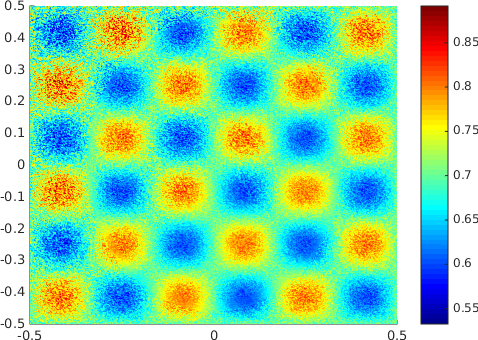}
	\includegraphics[width=0.3\textwidth,height=0.16\textheight]{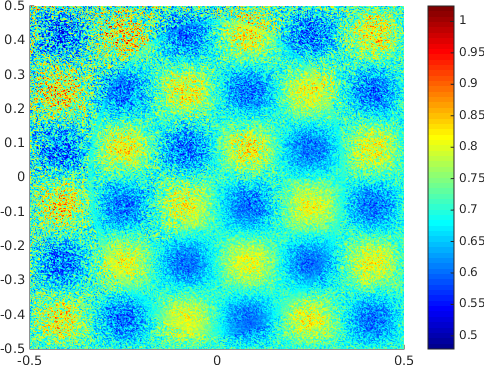}
	\caption{The reconstruction of $\sigma_{x,f}$ and $\eta$ in Case (1.1). First row, from left to right, $0\%$, $1\%$, $2\%$ random noises are added to the internal data $Q$ and the relative $L^1$ errors of the reconstructed $\sigma_{x,f}$ are $0.000132\%$, $3.88\%$, $7.76\%$ respectively. Second row, from left to right, assuming the knowledge of $\sigma_{x,f}$ from the first row, $0\%$, $1\%$, $2\%$ random noises are added to the internal data $S$.  The relative $L^2$ errors of reconstructed $\eta$ are $0.00313\%$, $5.60\%$, $11.7\%$ respectively.}
	\label{fig:1a}
\end{figure}

\textbf{Case (1.2): $\theta < -1$, $b \geq 0$ and $c \ge 0$.}
In this example, the intrinsic elasto-optical parameters are chosen as $n_x = -0.2$, $n_m = 0.5$, $n_f = -0.3$, then
\begin{equation}\nonumber
\begin{aligned}
&\gamma_x = -1.4,\quad \gamma_m = 0.0, \\
&\beta_x = 0.6,\quad \beta_m = 2.0,\quad \beta_f = 0.4,
\end{aligned}
\end{equation}
and $\tau = -3.5$, $\mu = 0.5$, $\theta = -\frac{9}{5}$.  The coefficients in~\eqref{eq:key} satisfy $b \geq 0$ and $c \ge 0$, 
hence we can apply Algorithm~\ref{alg:iter2} to solve the nonlinear equation.  The reconstructed images of $\sigma_{x,f}$  and $\eta$ are illustrated in Figure~\ref{fig:2a}.
\begin{figure}[!htb]
	\centering
	\includegraphics[width=0.3\textwidth,height=0.16\textheight]{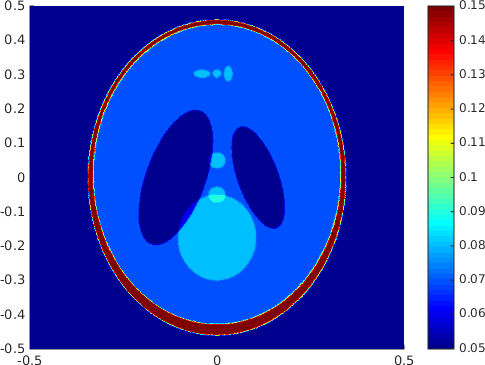}
	\includegraphics[width=0.3\textwidth,height=0.16\textheight]{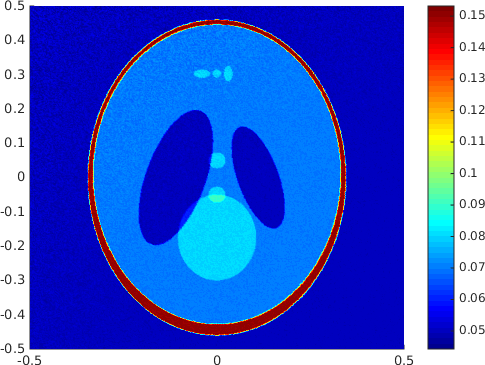}
	\includegraphics[width=0.3\textwidth,height=0.16\textheight]{d2e1fc}\\
	\includegraphics[width=0.3\textwidth,height=0.16\textheight]{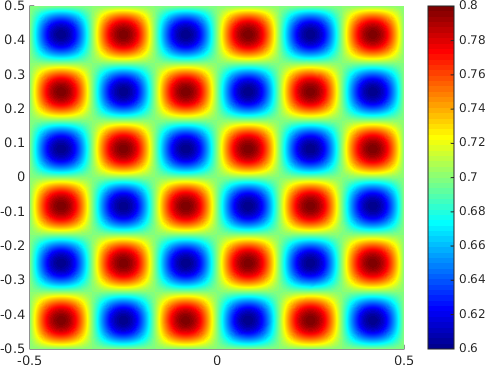}
	\includegraphics[width=0.3\textwidth,height=0.16\textheight]{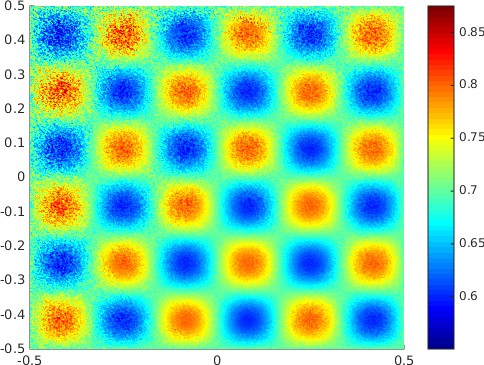}
	\includegraphics[width=0.3\textwidth,height=0.16\textheight]{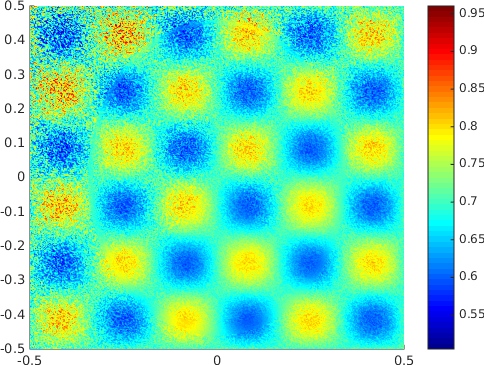}
	\caption{The reconstruction of $\sigma_{x,f}$ and $\eta$ in Case (1.2). First row, from left to right, $0\%$, $1\%$, $2\%$ random noises are added to the internal data $Q$ and the relative $L^1$ errors of reconstructed $\sigma_{x,f}$ are $0.0086\%$, $2.62\%$, $5.27\%$ respectively. Second row, from left to right, assuming the knowledge of $\sigma_{x,f}$ from the first row, $0\%$, $1\%$, $2\%$ random noises are added to the internal data $S$.  The relative $L^2$ errors of reconstructed $\eta$ are $0.0150\%$, $4.23\%$, $8.89\%$ respectively.}
	\label{fig:2a}
\end{figure}

\medskip
\textbf{Case (2): $\theta \geq 0$, $b \geq 0$ and $c \leq 0$.}
In this example, we choose the intrinsic elasto-optical parameters as $n_x = 0.6$, $n_m = 0.8$, $n_f = -0.65$, then
\begin{equation}\nonumber
\begin{aligned}
&\gamma_x = 0.2,\quad \gamma_m = 0.6, \\
&\beta_x = 2.2,\quad \beta_m = 2.6,\quad \beta_f = -0.3,
\end{aligned}
\end{equation}
and $\tau = -\frac{2}{3}$, $\mu = -\frac{25}{8}$, $\theta = 5$.  The coefficients in~\eqref{eq:key} satisfy $b \geq 0$ and $c \leq 0$. 
We apply Algorithm~\ref{alg:iter3} to solve the nonlinear equation.
The reconstructed images of $\sigma_{x,f}$  and $\eta$ are in Figure~\ref{fig:3a}.
\begin{figure}[!htb]
	\centering
	\includegraphics[width=0.3\textwidth,height=0.16\textheight]{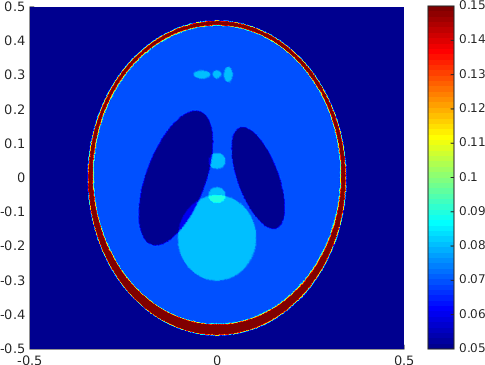}
	\includegraphics[width=0.3\textwidth,height=0.16\textheight]{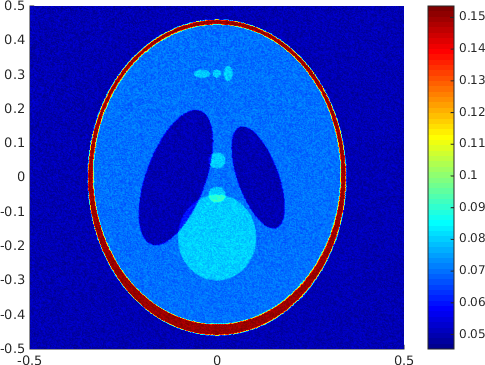}
	\includegraphics[width=0.3\textwidth,height=0.16\textheight]{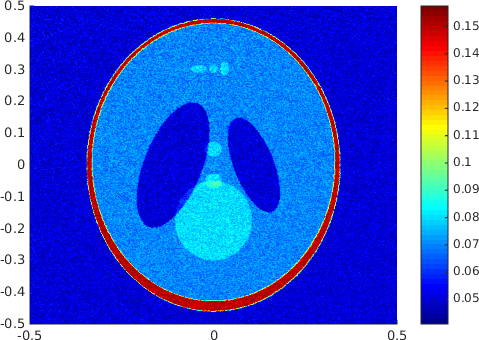}\\
	\includegraphics[width=0.3\textwidth,height=0.16\textheight]{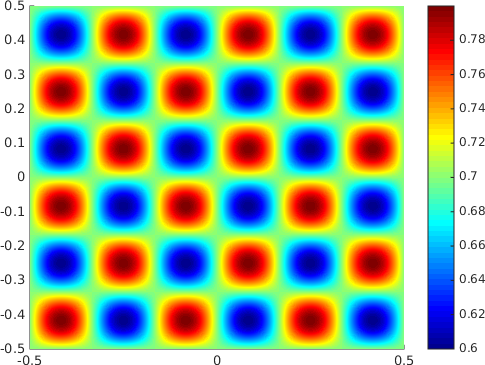}
	\includegraphics[width=0.3\textwidth,height=0.16\textheight]{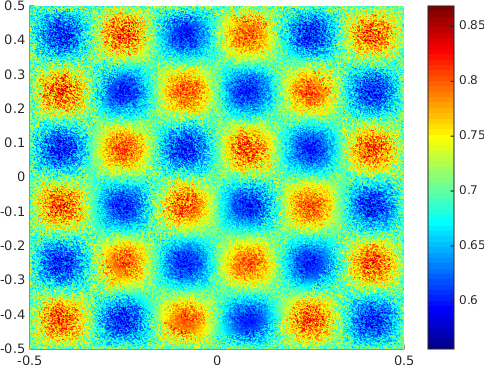}
	\includegraphics[width=0.3\textwidth,height=0.16\textheight]{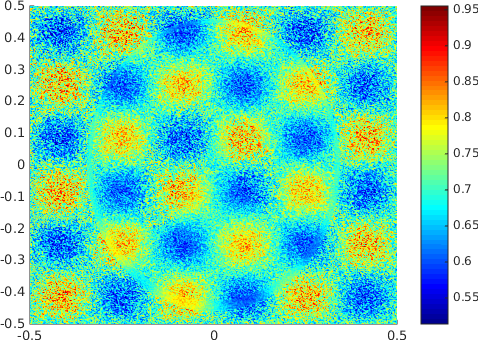}
	\caption{The reconstruction of $\sigma_{x,f}$ and $\eta$ in Example III. First row, from left to right, $0\%$, $1\%$, $2\%$ random noises are added to the internal data $Q$ and the relative $L^1$ errors of reconstructed $\sigma_{x,f}$ are $0.00147\%$, $3.68\%$, $7.38\%$ respectively. Second row, from left to right, assuming the knowledge of $\sigma_{x,f}$ from the first row, $0\%$, $1\%$, $2\%$ random noises are added to the internal data $S$.  The relative $L^2$ errors of reconstructed $\eta$ are $0.00392\%$, $4.65\%$, $9.48\%$ respectively.}
	\label{fig:3a}
\end{figure}
\section{Conclusion}\label{SEC:CONCLUSION}
In this paper, we studied the fluorescence ultrasound modulated optical tomography in the diffusive regime. Assuming knowledge of the elasto-optical coefficients and the background parameters $D_x, D_m, \sigma_{x,a}, \sigma_{m,a}$, we proved that the absorption coefficient $\sigma_{x,f}$ and the quantum efficiency coefficient $\eta$ can be uniquely and stably reconstructed from the internal data $Q$ in~\eqref{eq:hamilton1} and $S$ in~\eqref{eq:S} extracted from the boundary measurement~\eqref{eq:meas}. A key step is the analysis of an elliptic semi-linear equation in~\eqref{eq:key}. We proposed three iterative algorithms to solve this equation based on the signs of the coefficients. These algorithms are shown to generate sequences of functions which converge to the unique weak solution. Finally, reconstructive procedures for $\sigma_{s,f}$ and $\eta$ are provided and numerically implemented in several experiments, in the presence or absence of noise, to demonstrate the efficiency of the reconstruction.

\section*{Acknowledgement}


This material is based upon work supported by the National Science Foundation under Grant No. DMS-1439786 while the authors were in residence at the Institute of Computational and Experimental Research in Mathematics (ICERM) in Providence, RI, during the Fall 2017 semester program ``Mathematical Challenges in Radar and Seismic Imaging". The authors would like to thank ICERM for the hospitality and for creating a collaborative atmosphere. The authors are also grateful to Prof. Kui Ren from University of Texas Austin for bringing our attention to fUMOT.

\bibliographystyle{siam}
\bibliography{../bib/main}
\end{document}